\documentclass{birkjour}
\usepackage{amsmath}
\usepackage{amsfonts}
\usepackage{amssymb}
\usepackage[english]{babel}
\usepackage{color}
\usepackage{graphicx}
\usepackage{lmodern}
\usepackage{amsthm}
\usepackage{mathrsfs}
\usepackage{microtype}
\usepackage{mathscinet}
\usepackage{enumitem}
\usepackage[cal=boondoxo,bb=ams]{mathalfa}
\usepackage{hyperref}
\hypersetup{hidelinks}

\DeclareMathOperator\supp{supp}

\DeclareMathOperator\Gla{Gla}

\DeclareMathOperator\loc{loc}

 \newtheorem{thm}{Theorem}[section]
 
 \newtheorem{lem}[thm]{Lemma}
 
 \theoremstyle{definition}
 \newtheorem{defn}[thm]{Definition}
 \theoremstyle{remark}
 \newtheorem{rem}[thm]{Remark}
 
 \numberwithin{equation}{section}

\begin{document}

%
%
%
%
%
%
%
%
%

\title[Blow-up results for a system of semilinear damped wave equations]
 {Nonexistence of global solutions for a weakly coupled system of semilinear damped wave equations of derivative type in the scattering case}

\author[A. Palmieri]{Alessandro Palmieri}

\address{
Institute of Applied Analysis \\ Faculty for Mathematics and Computer Science \\ Technical University Bergakademie Freiberg\\
Pr\"{u}ferstra{\ss}e 9\\
 09596 Freiberg\\
Germany}

\email{alessandro.palmieri.math@gmail.com}

\author[H. Takamura]{Hiroyuki Takamura}
\address{Mathematical Institute \\ Tohoku University \\ Aoba \\ Sendai 980-8578 \\ Japan}
\email{hiroyuki.takamura.a1@tohoku.ac.jp}


\date{December 19, 2018}

\begin{abstract}
In this paper we consider the blow-up for solutions to a weakly coupled system of semilinear damped wave equations of derivative type in the scattering case. After introducing suitable functionals proposed by Lai-Takamura for the corresponding single semilinear equation, we employ Kato's lemma to derive the blow-up result in the subcritical case. On the other hand, in the critical case an iteration procedure based on the slicing method is employed. Let us point out that we find as critical curve in the $p$ - $q$ plane for the pair of exponents $(p,q)$ in the nonlinear terms the same one as for the weakly coupled system of semilinear not-damped wave equations with the same kind of nonlinearities.
\end{abstract}

\maketitle

\section{Introduction}

In this paper we consider a weakly coupled system of wave equations with time-dependent and scattering producing damping terms and with powers of the first order time-derivatives of components of the solution as nonlinear terms (semilinear term of \emph{derivative type}), namely, 
\begin{align}\label{weakly coupled system}
\begin{cases}
u_{tt}-\Delta u +b_1(t)u_t = |\partial_t v|^p,  & x\in \mathbb{R}^n, \ t>0,  \\
v_{tt}-\Delta v +b_2(t)v_t = |\partial_t u|^q,  & x\in \mathbb{R}^n, \ t>0, \\
 (u,u_t,v,v_t)(0,x)= (\varepsilon u_0, \varepsilon u_1, \varepsilon v_0, \varepsilon v_1)(x) & x\in \mathbb{R}^n,
\end{cases}
\end{align}
where $b_1,b_2\in \mathcal{C}([0,\infty))\cap L^1([0,\infty))$ are nonnegative functions, $\varepsilon$ is a positive parameter describing the size of initial data and $p,q>1$. We will prove blow-up results for \eqref{weakly coupled system} both in the subcritical case and in the critical case. Moreover, an upper bound for the lifespan of local solutions is derived in these two cases.

The nonexitence of global in time solutions in the case without damping terms (that is, for $b_1= b_2=0$) has been studied in \cite{Deng99,Xu04}, while the existence part has been proved in the three dimensional and radial case in \cite{KKS06}. Recently, in \cite[Section 8]{ISW18} the upper bound for the lifespan has been derived.

Summarizing the blow-up results above cited, for the weakly coupled system
\begin{align}\label{weakly coupled system classical wave}
\begin{cases}
u_{tt}-\Delta u  = |\partial_t v|^p,  & x\in \mathbb{R}^n, \ t>0,  \\
v_{tt}-\Delta v  = |\partial_t u|^q,  & x\in \mathbb{R}^n, \ t>0, \\
 (u,u_t,v,v_t)(0,x)= (\varepsilon u_0, \varepsilon u_1, \varepsilon v_0, \varepsilon v_1)(x) & x\in \mathbb{R}^n,
\end{cases}
\end{align} it holds the following result: under certain integral sign assumptions for the initial data, if $p,q>1$ satisfy 
\begin{align*}
\max\left\{\frac{p+1}{pq-1};\frac{q+1}{pq-1}\right\} \geqslant \frac{n-1}{2},
\end{align*} then, $(u,v)$ has to blow-up in finite time; moreover, the following upper bound estimate for the lifespan holds
\begin{align*} 
T(\varepsilon)\leqslant \begin{cases} C \varepsilon ^{-\max\{\Lambda(n,p,q),\Lambda(n,q,p)\}^{-1}} & \mbox{if} \ \ \Upsilon(n,p,q)>0, \\
\exp\big(C \varepsilon ^{-(pq-1)}\big) & \mbox{if} \ \ \Upsilon(n,p,q)=0, \, p \neq q,  \\
\exp\big(C \varepsilon ^{-(p-1)}\big) & \mbox{if} \ \ \Upsilon(n,p,q)=0, \, p=  q,\end{cases}
\end{align*} where
\begin{align}\label{def Lambda(n,p,q) function}
\Lambda (n,p,q)\doteq \frac{p+1}{pq-1}-\frac{n-1}{2}\, 
\end{align} and 
\begin{align} \label{def Upsilon(n,p,q) function}
\Upsilon(n,p,q)\doteq \max\{\Lambda(n,p,q),\Lambda(n,q,p)\}.
\end{align}

Let us stress that the study of the blow-up results for the system \eqref{weakly coupled system classical wave} is not a trivial generalization of the corresponding results related to Glassey's conjecture for the semilinear Cauchy problem
\begin{align}\label{semilinear Glassey problem}
\begin{cases}
u_{tt}-\Delta u  = |\partial_t u|^p,  & x\in \mathbb{R}^n, \ t>0, \\
 (u,u_t)(0,x)= (\varepsilon u_0, \varepsilon u_1)(x) & x\in \mathbb{R}^n.
\end{cases}
\end{align} Indeed, for the semilinear Cauchy problem \eqref{semilinear Glassey problem} it has been proved that the critical exponent is the so-called Glassey exponent 
\begin{align}\label{Glassey exponent def}
p_{\Gla}(n)\doteq \frac{n+1}{n-1}
\end{align}
 (in dimension $n=1$ there exist solutions that cannot be prolonged for all time for any exponent $p>1$ regardless of the smallness of initial data). 
 
 We refer to the classical results \cite{Joh81,Sid83,Mas83,Sch85,Ram87,Age91,HT95,Tzv98,Zhou01,HWY12} and references therein contained for further details.  
 
 We remark that the condition $p\leqslant p_{\Gla}(n)$ can be equivalently expressed as 
\begin{align} \label{equiv cond below glassey exponent}
\frac{1}{p-1}\geqslant \frac{n-1}{2}.
\end{align} Therefore, because of
\begin{align} \label{inequality}
\frac{\max\{p,q\}+1}{pq-1} \geqslant \frac{1}{\max\{p,q\}-1},
\end{align} where the equality holds if and only if $p=q$, it may happen that the condition for $(p,q)$, which implies the validity of a blow-up result, is satisfied even though one among $p,q$ is greater than the Glassey exponent (of course, in the case $p\neq q$). This fact follows immediately by \eqref{equiv cond below glassey exponent} and \eqref{inequality}.

Recently, semilinear wave equations with scattering producing damping terms have been studied in \cite{LT18Scatt,LT18Glass,LT18ComNon} in the case of single equations and in \cite{PalTak19,PalTak19mix} for weakly coupled systems with power nonlinearities and mixed nonlinearities, respectively.

In this work we will study blow-up results for the weakly coupled system \eqref{weakly coupled system} in the subcritical case and in the critical case by considering the blow-up dynamic of suitable functionals, that represent a generalization of the functional introduced in \cite{LT18Glass} in order to study the semilinear wave equation with damping in the scattering case related to Glassey's conjecture. 

The novelty of our results consists on the way in which methods, typically used for the semilinear classical wave equation with power nonlinearity of the solution itself, are suitably adapted to the study of \eqref{weakly coupled system}. More specifically, these methods are Kato's lemma (see \cite{Sid84,Yag05,Tak15,YZ06}) and the slicing method combined with an iteration argument (see \cite{AKT00,WakYor18}). These methods have been studied only for the classical semilinear wave equation with power nonlinearity $|u|^p$ and for the corresponding weakly coupled system. In this sense, it is surprising that they can be applied to study the weakly coupled system \eqref{weakly coupled system} with semilinear terms of derivative type.

Before stating the blow-up results of this paper, let us introduce a suitable notion of energy solutions.

\begin{defn} \label{def energ sol intro} Let $u_0,v_0\in H^1(\mathbb{R}^n)$ and $u_1,v_1\in L^2(\mathbb{R}^n)$.
We say that $(u,v)$ is an energy solution of \eqref{weakly coupled system} on $[0,T)$ if 
\begin{align*}
& u\in \mathcal{C}([0,T),H^1(\mathbb{R}^n))\cap \mathcal{C}^1([0,T),L^2(\mathbb{R}^n))\quad \mbox{and} \quad  \partial_t u\in L^q_{\loc}([0,T)\times\mathbb{R}^n), \\
& v\in \mathcal{C}([0,T),H^1(\mathbb{R}^n))\cap \mathcal{C}^1([0,T),L^2(\mathbb{R}^n)) \quad \mbox{and} \quad  \partial_t v\in  L^p_{\loc}([0,T)\times\mathbb{R}^n)
\end{align*}
satisfy $u(0,x)=\varepsilon u_0(x)$, $v(0,x)=\varepsilon v_0(x)$ in $H^1(\mathbb{R}^n)$ and the equalities
\begin{align} 
& \int_0^t \int_{\mathbb{R}^n}|\partial_tv(s,x)|^p\phi(s,x)\,dx \, ds  = \int_{\mathbb{R}^n}  \partial_t u(t,x)\phi(t,x)\,dx \notag \\
& \  -\int_{\mathbb{R}^n}\varepsilon u_1(x)\phi(0,x)\,dx - \int_0^t\int_{\mathbb{R}^n} \partial_t u(s,x)\phi_s(s,x) \,dx\, ds  \notag \\
& \ +\int_0^t\int_{\mathbb{R}^n}\nabla u(s,x)\cdot\nabla\phi(s,x)\, dx\, ds  +\int_0^t\int_{\mathbb{R}^n}b_1(s)\partial_t u(s,x) \phi(s,x)\,dx\, ds    \label{def u}
\end{align} and 
\begin{align} 
& \int_0^t \int_{\mathbb{R}^n}|\partial_t u(s,x)|^q\psi(s,x)\,dx \, ds = \int_{\mathbb{R}^n}  \partial_t v(t,x)\psi(t,x)\,dx \notag \\
& \ -\int_{\mathbb{R}^n}\varepsilon v_1(x)\psi(0,x)\,dx - \int_0^t\int_{\mathbb{R}^n} \partial_t v(s,x)\psi_s(s,x) \,dx\, ds   \notag \\
& \ +\int_0^t\int_{\mathbb{R}^n}\nabla v(s,x)\cdot\nabla\psi(s,x)\, dx\, ds+\int_0^t\int_{\mathbb{R}^n} b_2(s) \partial_t v(s,x)\psi(s,x)\,dx\, ds   \label{def v}
\end{align}
for any test functions $\phi,\psi \in \mathcal{C}_0^\infty([0,T)\times\mathbb{R}^n)$ and any $t\in [0,T)$.
\end{defn}

 Performing a further step of integrations by parts in \eqref{def u}, \eqref{def v} and letting $t\rightarrow T$, we find that $(u,v)$ fulfills the definition of weak solution to \eqref{weakly coupled system}.
 
Let us state the main blow-up result for \eqref{weakly coupled system} of this paper.

\begin{thm}\label{Thm blowup |vt|^p, |ut|^q} Let $b_1,b_2$ be continuous, nonnegative and summable functions. Let us consider $p,q>1$ satisfying 
\begin{align}\label{critical exponent wave like case system}
\max\left\{\frac{p+1}{pq-1},\frac{q+1}{pq-1}\right\}\geqslant \frac{n-1}{2}\,.
\end{align} 

Assume that $u_0,v_0\in H^1(\mathbb{R}^n)$ and $u_1,v_1\in  L^2(\mathbb{R}^n)$ are nonnegative and compactly supported in $B_R$ 
functions such that $u_1,v_1\not \equiv 0$. 


Let $(u,v)$ be an energy solution of \eqref{weakly coupled system} with lifespan $T=T(\varepsilon)$ according to Definition \ref{def energ sol intro}, satisfying
\begin{align}\label{support property solution}
\supp u, \, \supp v \subset \{ (t,x)\in [0,T)\times \mathbb{R}^n: |x|\leqslant t+R\}.
\end{align}
 Then, there exists a positive constant $\varepsilon_0=\varepsilon_0(u_0,u_1,v_0,v_1,n,p,q,b_1,b_2,R)$ such that for any $\varepsilon\in (0,\varepsilon_0]$ the solution $(u,v)$ blows up in finite time. Moreover,
 the upper bound estimate for the lifespan
\begin{align} \label{lifespan upper bound estimate}
T(\varepsilon)\leqslant \begin{cases} C \varepsilon ^{-\max\{\Lambda(n,p,q),\Lambda(n,q,p)\}^{-1}} & \mbox{if} \ \ \Upsilon (n,p,q)>0, \\
\exp\big(C \varepsilon ^{-(pq-1)}\big) & \mbox{if} \ \ \Upsilon(n,p,q)=0, \, p \neq q,  \\
\exp\big(C \varepsilon ^{-(p-1)}\big) & \mbox{if} \ \ \Upsilon(n,p,q)=0, \, p=  q,\end{cases}
\end{align} holds, where C is an independent of $\varepsilon$, positive constant and $\Lambda (n,p,q)$ and $\Upsilon(n,p,q)$ are defined by \eqref{def Lambda(n,p,q) function} and \eqref{def Upsilon(n,p,q) function}.
\end{thm}

\begin{rem} The upper bound estimates in \eqref{lifespan upper bound estimate} for the lifespan of the solution coincide with the ones for the case  $b_1=b_2=0$ as we recalled in the historical overview in the first part of this introduction.
\end{rem}

The remaining part of this paper is organized as follows: in Section \ref{Section Functionals F and G} we recall the definition of a multiplier, that has been introduced in \cite{LT18Scatt} in order to study the corresponding single semilinear wave equation with power nonlinearity, and its properties; moreover, following \cite{LT18Glass} we introduce a suitable pair of functionals related to a local solution of \eqref{weakly coupled system} and we determine certain lower bounds for these functionals; then, in Section \ref{Section |vt|^p, |ut|^q subcrical case} we prove Theorem \ref{Thm blowup |vt|^p, |ut|^q} in the subcritical case  by using a Kato's type lemma. Finally, in Section \ref{Section |vt|^p, |ut|^q crical case} we modify the approach in the critical case employing an iteration argument together with the slicing method.

\subsection*{Notations} Throughout this paper we will use the following notations: $B_R$ denotes the ball around the origin with radius $R$; $f \lesssim g$ means that there exists a positive constant $C$ such that $f \leqslant Cg$ and, analogously, for $f\gtrsim g$; 
finally, as in the introduction, $p_{\Gla}(n)$ denotes the Glassey exponent, whose definition is given in \eqref{Glassey exponent def}. 

\section{Definition of the functionals and derivation of the iteration frame}
\label{Section Functionals F and G}

Let us recall the definition of some multipliers related to our model, which have been introduced in \cite{LT18Scatt}, and some of their properties as well, that will be useful throughout this paper.
 
\begin{defn} Let $b_1,b_2\in \mathcal{C}([0,\infty))\cap L^1([0,\infty))$ be the nonnegative, time-dependent coefficients in \eqref{weakly coupled system}. We define the corresponding \emph{multipliers}
\begin{align*}
m_j(t)\doteq \exp \bigg(-\int_t^\infty b_j(\tau) d\tau \bigg) \qquad \mbox{for} \ \  t\geqslant 0 \ \ \mbox{and} \ j=1,2. 
\end{align*}
\end{defn}

Due to the nonnegativity of $b_1,b_2$, it follows the monotonicity of $m_1,m_2$. Furthermore, as these coefficients are summable, we get also that these multipliers are bounded and 
\begin{align}
m_j(0)\leqslant m_j(t) \leqslant 1 \qquad \mbox{for} \ \  t\geqslant 0 \ \ \mbox{and} \ j=1,2. \label{boundedness multipliers}
\end{align}

An important property of these multipliers is the relation with the corresponding derivatives. More precisely,
\begin{align}
m_j'(t) = b_j(t) \, m(t)  \qquad \mbox{for} \ \   j=1,2. \label{derivative multiplier}
\end{align} The properties described by \eqref{boundedness multipliers} and \eqref{derivative multiplier}  play a crucial role in the remaining part of this section, which  is devoted to the introduction of a pair of functionals, whose dynamic is investigated in the proof of Theorem \ref{Thm blowup |vt|^p, |ut|^q}. This kind of functionals have been considered for a single semilinear wave equation of derivative type with a scattering-producing damping in \cite{LT18Glass}.

However, before introducing the above quoted functionals, we need to derive suitable lower bound estimates for a different pair of functionals related to a local solution $(u,v)$ of \eqref{weakly coupled system}. Thus, we introduce the functionals
\begin{align}
U_1(t)\doteq \int_{\mathbb{R}^n}u(t,x)\Psi(t,x)\, dx  \quad  \mbox{and} \quad V_1(t)\doteq \int_{\mathbb{R}^n}v(t,x)\Psi(t,x)\, dx , \label{def U1 and V1}
\end{align}
where $\Psi=\Psi(t,x)\doteq e^{-t} \Phi(x)$ and 
\begin{align}
\Phi=\Phi(x)\doteq \begin{cases} e^{x}+e^{-x} & \mbox{for} \ \ n=1, \\\displaystyle{\int_{\mathbb{S}^{n-1}} \, e^{\omega \cdot x}\, dS_{\omega}} & \mbox{for} \ \ n\geqslant 2\end{cases} \label{def eigenfunction laplace op}
\end{align} is an eigenfunction of the Laplace operator, as $\Delta \Phi =\Phi$. In order to derive lower bounds for $U_1,V_1$ we prove a result, which is valid even when we consider more general nonnegative nonlinearities.


\begin{lem} \label{lemma U1,V1} Let $(w,\widetilde{w})$ be a local energy solution of the Cauchy problem
\begin{align*}
\begin{cases}
w_{tt}-\Delta w +b_1(t)w_t = G_1(t,x,w,w_t,\widetilde{w},\widetilde{w}_t),  & x\in \mathbb{R}^n, \ t\in (0,T),  \\
\widetilde{w}_{tt}-\Delta \widetilde{w} +b_2(t) \widetilde{w}_t = G_2(t,x,w,w_t,\widetilde{w},\widetilde{w}_t),  & x\in \mathbb{R}^n, \ t\in (0,T), \\
 (w,w_t,\widetilde{w},\widetilde{w}_t)(0,x)= (\varepsilon w_0, \varepsilon w_1, \varepsilon \widetilde{w}_0, \varepsilon \widetilde{w}_1)(x), & x\in \mathbb{R}^n,
\end{cases}
\end{align*} where the coefficients of the damping terms $b_1,b_2\in \mathcal{C}([0,\infty))\cap L^1([0,\infty))$ and the nonlinear terms $G_1,G_2$ are nonnegative. Furthermore, we assume that $w_0,w_1,\widetilde{w}_0,\widetilde{w}_1$ are nonnegative, pairwise nontrivial and compactly supported and that $w,\widetilde{w}$ satisfy a support condition as in \eqref{support property solution}.
 Let $W_1,\widetilde{W}_1$ be defined by $$ W_1(t) \doteq \int_{\mathbb{R}^n}w(t,x)\Psi(t,x)\, dx  \quad \mbox{and} \quad
\widetilde{W}_1(t) \doteq \int_{\mathbb{R}^n}\widetilde{w}(t,x)\Psi(t,x)\, dx  $$ for any $t\geqslant 0$. Then, for any $t\geqslant 0$ the following estimates hold
\begin{align*}
W_1(t) & \geqslant \varepsilon \,\frac{m_1(0)}{2}  \int_{\mathbb{R}^n}w_0(x) \Phi(x) \, dx, \\ 
\widetilde{W}_1(t)  & \geqslant \varepsilon \,\frac{m_2(0)}{2} \int_{\mathbb{R}^n}\widetilde{w}_0(x) \Phi(x) \, dx. 
\end{align*}
\end{lem}

\begin{proof}
We follow the main ideas of the proof of Lemma 3.1 in \cite{LT18Glass}. We prove the lower bound estimate for $W_1$, since the proof of the one for $ \widetilde{W}_1$ is completely analogous. Clearly, the definition of energy solution for the considered Cauchy problem is analogous to the one given in Definition \ref{def energ sol intro} for \eqref{weakly coupled system}. The only difference consists in the assumptions on the nonlinear terms $G_1,G_2$, which have to be supposed in $L^1_{\loc}([0,T)\times\mathbb{R}^n)$ and replace $|v_t|^p,|u_t|^q$ in \eqref{def u}, \eqref{def v}, respectively.

Thanks to the support property for $w$, we can apply the definition of energy solution with test functions that are not compactly supported. Hence, employing the definition of energy solution with $\Psi$ as test function and taking the time derivative of the obtained relation, we find for any $t\in (0,T)$ 
\begin{align*}
 & \frac{d}{dt}  \int_{\mathbb{R}^n} \partial_t w(t,x)\Psi(t,x) \, dx - \int_{\mathbb{R}^n}\Big(\partial_tw(t,x)\Psi_t(t,x)+ w(t,x)\Delta \Psi(t,x)\Big)dx \\ & \quad   +\int_{\mathbb{R}^n}b_1(t)\partial_t w(t,x) \Psi(t,x) dx  	 = \int_{\mathbb{R}^n}G_1(t,x)\Psi(t,x) \,dx,
\end{align*}  where we denote $ G_1(t,x)\equiv G_1(t,x,w(t,x),w_t(t,x),\widetilde{w}(t,x),\widetilde{w}_t(t,x))$ for the sake of brevity.
Multiplying both sides of the previous equality by $m_1(t)$, we find
\begin{align*}
 m_1(t) &  \int_{\mathbb{R}^n}G_1(t,x)\Psi(t,x) dx  \\ & =  m_1(t)\frac{d}{dt}  \int_{\mathbb{R}^n} \partial_t w(t,x)\Psi(t,x) \, dx + m_1(t)b_1(t) \int_{\mathbb{R}^n} \partial_t w(t,x) \Psi(t,x) \,dx \\ & \quad - m_1(t) \int_{\mathbb{R}^n}\big(\partial_t w(t,x)\Psi_t(t,x)+w(t,x)\, \Delta \Psi(t,x)\big)\, dx  \\
& =  \frac{d}{dt} \bigg( m_1(t) \!\int_{\mathbb{R}^n} \partial_t w(t,x)\Psi(t,x) \, dx \bigg) \\ & \quad +m_1(t) \int_{\mathbb{R}^n}\!\big(\partial_t w(t,x) -w(t,x)\big)\Psi(t,x) \, dx,
\end{align*}
where in the last step we used \eqref{derivative multiplier} and the properties $\Psi_t=-\Psi$ and $\Delta \Psi= \Psi$.
Integrating the last relation over $[0,t]$, we find 
\begin{align}
\int_0^t m_1(s) & \int_{\mathbb{R}^n} G_1(s,x)\Psi(s,x) \, dx\,  ds \notag  \\ \quad  & =  m_1(t) \int_{\mathbb{R}^n} \partial_t w(t,x)\Psi(t,x) \, dx -\varepsilon \, m_1(0) \int_{\mathbb{R}^n} w_1(x)\, \Phi(x) \, dx \notag \\
 & \qquad + \int_0^t m_1(s)\int_{\mathbb{R}^n}\big(\partial_t w(s,x) -w(s,x)\big)\Psi(s,x)\, dx \, ds. \label{inter A}
\end{align}
An integration by parts with respect to $t$ provides
\begin{align}
 \int_0^t   m_1(s) & \int_{\mathbb{R}^n} \partial_t w(s,x)  \Psi(s,x) \,  dx \, ds  \notag \\ 
& = 
m_1(t)\int_{\mathbb{R}^n} w(t,x) \Psi(t,x) \,  dx - \varepsilon \, 
m_1(0)\int_{\mathbb{R}^n} w_0(x)\, \Phi(x) \,  dx \notag  \\  & \quad - \int_0^t \int_{\mathbb{R}^n} w(s,x) \, b_1(s)\,  m_1(s) \Psi(s,x) \,  dx \, ds \notag  \\  & \quad  +\int_0^t  m_1(s) \int_{\mathbb{R}^n} w(s,x)   \Psi(s,x) \,  dx \, ds. \label{inter B}
\end{align} 

Consequently, combining \eqref{inter A} and \eqref{inter B}, we arrive at
\begin{align*}
\int_0^t & m_1(s) \int_{\mathbb{R}^n}G_1(s,x)\Psi(s,x) \, dx\,  ds \\ & \quad \quad  + \int_0^t b_1(s)\,  m_1(s) \int_{\mathbb{R}^n} w(s,x) \,  \Psi(s,x) \,  dx \, ds  \\ & \quad \quad + \varepsilon \, m_1(0) \int_{\mathbb{R}^n} \big(w_0(x)+w_1(x)\big) \Phi(x) \, dx   \\ & \quad  =  m_1(t) \int_{\mathbb{R}^n} \big( \partial_t w(t,x)\Psi(t,x) +w(t,x) \Psi(t,x)\big) \, dx  
\\ & \quad  =  m_1(t) \, \frac{d}{dt} \int_{\mathbb{R}^n}  w(t,x)\Psi(t,x) \, dx  + 2m_1(t) \int_{\mathbb{R}^n} w(t,x) \Psi(t,x) \, dx.  
\end{align*}
By the definition of $W_1$ and the nonnegativity of the semilinear term $G_1$, from the previous relation we obtain the inequality
\begin{align*}
m_1(t) \big(W'_1(t)+2 W_1(t)\big) \geqslant \varepsilon \, m_1(0) \, C(w_0,w_1) + \int_0^t b_1(s)\,  m_1(s) \, W_1(s) \, ds 
\end{align*} where $C(w_0,w_1)\doteq  \int_{\mathbb{R}^n} \big(w_0(x)+w_1(x)\big) \Phi(x) \, dx $. Since the multiplier $m_1$ is bounded, using \eqref{boundedness multipliers}, we have
\begin{align}
W'_1(t)+2 W_1(t)  
&\geqslant \varepsilon \, m_1(0) \, C(w_0,w_1) +\frac{1}{m(t)} \int_0^t b_1(s)\,  m_1(s)\,  W_1(s) \, ds. \label{diff inequality U1}
\end{align} Multiplying  both sides in the last estimate by $e^{2t}$ and integrating over $[0,t]$, we get 
\begin{align}
e^{2t}W_1(t)& \geqslant W_1(0)+ \frac{m_1(0)}{2} \, \varepsilon \,  C(w_0,w_1) (e^{2t}-1)\notag \\ & \qquad +\int_0^t  \frac{e^{2s}}{m(s)} \int_0^\tau b_1(\tau)\,  m_1(\tau) \, W_1(\tau) \, d\tau  \, ds. \label{comparison}
\end{align}
Let us prove first the positiveness of $W_1$ by using a comparison argument. Because we assumed the data pairwise nontrivial,  at least one among  $w_0,w_1$ is not identically $0$. In the first case $w_0\not \equiv 0$, as $w_0\geqslant 0$ implies $W_1(0)>0$, by continuity it holds $W_1(t)>0$ in a right neighborhood of $t=0$. If $t_0>0$ was the smallest time such that $W_1(t_0)=0$, then, the evaluation of \eqref{comparison} in $t=t_0$ would yield a contradiction. In the second case $w_0\equiv 0$ and $w_1\not \equiv 0$, we can use \eqref{diff inequality U1} to find a contradiction. In fact, in this case $W_1(0)=0$ and $W_1'(0)=\varepsilon \int_{\mathbb{R}^n} w_1(x)\Phi(x)\, dx >0$. Thus, by continuity $W_1'(t)>0$ for any $t\in [0,t_1)$ for a suitable $t_1>0$. Hence, $W_1$ is strictly increasing, and, in particular, positive in $(0,t_1)$. We assume by contradiction that $t_2>t_1$ is the smallest time such that $W_1(t_2)=0$. Then, $W_1'(t_2)\leqslant 0$. Indeed, if $W_1'(t_2)$ was positive, then, $W_1$ would be strictly increasing in a neighborhood of $t_2$, but this would be a contradiction to the definition of $t_2$. In fact, there would be a smaller zero, because $W_1$ would be negative in a left neighborhood of $t_2$. Finally, if we plug $W_1(t_2)=0, W_1'(t_2)\leqslant 0$ and we use $W_1(t)>0$ for any $t\in (0,t_2)$ in \eqref{diff inequality U1}, we have a contradiction.

 Also, thanks to the fact that $W_1$ is positive, from \eqref{comparison} we obtain
\begin{align*}
W_1(t) &\geqslant e^{-2t} W_1(0)+ \frac{m_1(0)}{2}  \, \varepsilon \,  C(w_0,w_1) (1-e^{-2t}) \\ & \geqslant \frac{m_1(0)}{2} \,\varepsilon \int_{\mathbb{R}^n}w_0(x) \Phi(x) \, dx , 
\end{align*} which is the desired estimate. This concludes the proof. 
\end{proof}

In particular, applying Lemma \ref{lemma U1,V1} to an energy  solution $(u,v)$ of \eqref{weakly coupled system}, we find
\begin{align}
U_1(t) & \geqslant \frac{m_1(0)}{2} \,\varepsilon \int_{\mathbb{R}^n}u_0(x) \Phi(x) \, dx \label{lower bound U1}, \\
V_1(t) & \geqslant \frac{m_2(0)}{2} \, \varepsilon \int_{\mathbb{R}^n}v_0(x) \Phi(x) \, dx \label{lower bound V1}
\end{align}
for any $t\geqslant 0$, where $U_1$ and $V_1$ are defined by \eqref{def U1 and V1}.

Next we follow the main ideas of \cite[Section 3]{LT18Glass} in order to introduce the suitable functionals for the proof of the blow-up result.
Let us point out that 
\begin{align}
\frac{d}{dt}& \bigg( m_1(t) \int_{\mathbb{R}^n}\big(\partial_t u(t,x)+u(t,x)\big) \Psi(t,x) \, dx \bigg) \notag\\ &= b_1(t) m_1(t) \int_{\mathbb{R}^n}\big(\partial_t u(t,x)+u(t,x)\big) \Psi(t,x) \, dx \notag\\ & \quad +m_1(t) \frac{d}{dt}  \int_{\mathbb{R}^n}\big(\partial_t u(t,x)+u(t,x)\big) \Psi(t,x) \, dx. \label{inter 1}
\end{align} 
Choosing $\phi\equiv \Psi$ on $\supp u$ in \eqref{def u}, we have
\begin{align*}
& \int_{\mathbb{R}^n}  \partial_t u(t,x)\Psi(t,x)\,dx-\int_{\mathbb{R}^n}\varepsilon u_1(x)\Phi(x)\,dx - \int_0^t\int_{\mathbb{R}^n} \partial_t u(s,x)\Psi_s(s,x) \,dx\, ds\notag \\
& \ \  +\int_0^t\int_{\mathbb{R}^n}\nabla u(s,x)\cdot\nabla\Psi(s,x)\, dx\, ds+\int_0^t\int_{\mathbb{R}^n}b_1(s)\partial_t u(s,x) \Psi(s,x)\,dx\, ds  \\ & =\int_0^t \int_{\mathbb{R}^n}|\partial_t v(s,x)|^p\Psi(s,x)\,dx \, ds.
\end{align*} Differentiating both sides of the previous equality with respect to $t$, we arrive at
\begin{align}
\int_{\mathbb{R}^n} & \!|\partial_t v(t,x)|^p\Psi(t,x)\,dx  = \frac{d}{dt} \int_{\mathbb{R}^n}\! \partial_t u(t,x)\Psi(t,x)\, dx -\int_{\mathbb{R}^n} \! \partial_t u(t,x)\Psi_t(t,x)\,dx  \notag \\ & \quad + \int_{\mathbb{R}^n} \big( \nabla u(t,x)\cdot\nabla\Psi(t,x) + b_1(t)\partial_t u(t,x) \Psi(t,x)\big)  \,dx . \label{inter a} 
\end{align} Using $\Delta \Psi=\Psi$ and $\Psi_t=-\Psi$, \eqref{inter a} yields
\begin{align}
\int_{\mathbb{R}^n}|\partial_t v(t,x)|^p\Psi(t,x)\,dx &  = \frac{d}{dt} \int_{\mathbb{R}^n} \big( \partial_t u(t,x)+u(t,x)\big)\Psi(t,x)\, dx \notag\\ & \qquad +b_1(t) \int_{\mathbb{R}^n} \partial_t u(t,x) \Psi(t,x) \,dx. \label{inter 3}
\end{align} If we combine \eqref{inter 1} and \eqref{inter 3}, we obtain 
\begin{align}
 \frac{d}{dt} \bigg(  m_1(t) & \int_{\mathbb{R}^n}\big(\partial_t u(t,x) +u(t,x)\big) \Psi(t,x) \, dx \bigg) \notag  \\ &= b_1(t) m_1(t)\, U_1(t)  + m_1(t) \int_{\mathbb{R}^n}|\partial_t v(t,x)|^p\Psi(t,x)\,dx , 
\label{inter 5}
\end{align} where $U_1$ is defined by \eqref{def U1 and V1}. Thanks to \eqref{lower bound U1} we have that $U_1$ is nonnegative, then, integrating \eqref{inter 5} over $[0,t]$, we get  the estimate
\begin{align}
m_1(t)  \int_{\mathbb{R}^n}\big(\partial_t u(t,x) &  +u(t,x)\big) \Psi(t,x)  \, dx  \notag\\ & \geqslant \varepsilon \, m_1(t) \int_{\mathbb{R}^n}\big(u_0(x) +u_1(x)\big) \Phi(x) \, dx   \notag\\ & \quad  + \int_0^t m_1(s) \int_{\mathbb{R}^n}|\partial_t v(s,x)|^p\Psi(s,x)\,dx. \label{inter 7}
\end{align} 
Furthermore, we may rewrite \eqref{inter a} as follows
\begin{align}
\int_{\mathbb{R}^n}|\partial_t v(t,x)|^p\Psi(t,x)\,dx & = \frac{d}{dt} \int_{\mathbb{R}^n} \partial_t u(t,x)\Psi(t,x)\, dx   \notag \\ & \qquad + b_1(t) \int_{\mathbb{R}^n} \partial_t u(t,x) \Psi(t,x)  \,dx \notag \\ & \qquad + \int_{\mathbb{R}^n} \big( \partial_t u(t,x)- u(t,x)\big)\Psi(t,x)\, dx.
\label{inter c}
\end{align} If we multiply both sides of \eqref{inter c} by $m_1(t)$, we find
\begin{align}
 \frac{d}{dt} \bigg(m_1(t)\int_{\mathbb{R}^n} \partial_t u(t,x) & \Psi(t,x) \bigg)  +m_1(t) \int_{\mathbb{R}^n} \big( \partial_t u(t,x)- u(t,x)\big)\Psi(t,x)\, dx \notag \\ &=  m_1(t) \int_{\mathbb{R}^n}|\partial_t v(t,x)|^p\Psi(t,x)\,dx.
\label{inter 9}
\end{align}
Adding \eqref{inter 7} and \eqref{inter 9}, we find
\begin{align}
& \frac{d}{dt} \bigg(m_1(t)\int_{\mathbb{R}^n}  \partial_t u(t,x)  \Psi(t,x) \, dx\bigg)   + 2m_1(t) \int_{\mathbb{R}^n}  \partial_t u(t,x)\Psi(t,x)\, dx  \notag \\ & \quad \geqslant  \varepsilon\, m_1(0) \int_{\mathbb{R}^n}\big(u_0(x) +u_1(x)\big) \Phi(x) \, dx + m_1(t) \int_{\mathbb{R}^n}|\partial_t v(t,x)|^p\Psi(t,x)\,dx  \notag \\ & \qquad + \int_0^t m_1(s) \int_{\mathbb{R}^n}|\partial_t v(s,x)|^p\Psi(s,x)\,dx. \label{inter 11}
\end{align}
In a complete analogous way, one can prove
\begin{align}
 & \frac{d}{dt} \bigg(m_2(t)\int_{\mathbb{R}^n}  \partial_t v(t,x)  \Psi(t,x) \, dx\bigg)   + 2m_2(t) \int_{\mathbb{R}^n}  \partial_t v(t,x)\Psi(t,x)\, dx  \notag \\ & \quad \geqslant  \varepsilon\, m_2(0) \int_{\mathbb{R}^n}\big(v_0(x) +v_1(x)\big) \Phi(x) \, dx + m_2(t) \int_{\mathbb{R}^n}|\partial_t u(t,x)|^q\Psi(t,x)\,dx  \notag \\ & \qquad + \int_0^t m_2(s) \int_{\mathbb{R}^n}|\partial_t u(s,x)|^q\Psi(s,x)\,dx. \label{inter 12}
\end{align}

Let us set the auxiliary functionals
\begin{align*}
U_2(t) & \doteq m_1(t) \int_{\mathbb{R}^n}  \partial_t u(t,x)  \Psi(t,x) \, dx -\varepsilon\, \frac{m_1(0)}{2} \int_{\mathbb{R}^n}u_1(x) \Phi(x) \, dx \\ & \qquad -\frac{1}{2} \int_0^t m_1(s) \int_{\mathbb{R}^n}|\partial_t v(s,x)|^p\Psi(t,x)\,dx \, ds \, ,  \\
V_2(t) & \doteq m_2(t) \int_{\mathbb{R}^n}  \partial_t v(t,x)  \Psi(t,x) \, dx -\varepsilon\, \frac{m_2(0)}{2} \int_{\mathbb{R}^n}v_1(x) \Phi(x) \, dx \\ & \qquad -\frac{1}{2} \int_0^t m_2(s) \int_{\mathbb{R}^n}|\partial_t u(s,x)|^q\Psi(t,x)\,dx \, ds \, .
\end{align*}
Clearly, 
\begin{align*}
U_2(0) &= \varepsilon\, \frac{m_1(0)}{2} \int_{\mathbb{R}^n}u_1(x) \Phi(x) \, dx\geqslant 0 , \\  V_2(0) &= \varepsilon\, \frac{m_2(0)}{2} \int_{\mathbb{R}^n}v_1(x) \Phi(x) \, dx \geqslant 0.
\end{align*} Besides, \eqref{inter 11} implies
\begin{align}
& U_2'(t)+2 U_2(t)  \notag \\ & \quad\geqslant \varepsilon\, m_1(0) \int_{\mathbb{R}^n}u_0(x) \Phi(x) \, dx   + \frac{1}{2}\, m_1(t) \int_{\mathbb{R}^n}  |\partial_t v(t,x)|^p  \Psi(t,x) \, dx \notag \\ & \quad \geqslant 0.\label{inter 13}
\end{align} Hence, multiplying the left hand side of \eqref{inter 13} by $e^{2t}$ and integrating over $[0,t]$, we get $U_2(t)\geqslant e^{-2t} U_2(0)\geqslant 0$. Similarly, employing \eqref{inter 12}, we can prove that $V_2(t)\geqslant 0$ for any $t\geqslant 0$.

Therefore, as the functionals $U_2,V_2$ are nonnegative we may write
\begin{align}
 m_1(t) \int_{\mathbb{R}^n}  \partial_t u(t,x) & \Psi(t,x) \, dx  \geqslant \varepsilon\, \frac{m_1(0)}{2} \int_{\mathbb{R}^n}u_1(x) \Phi(x) \, dx \notag  \\ & \quad +\frac{1}{2} \int_0^t m_1(s) \int_{\mathbb{R}^n}|\partial_t v(s,x)|^p\Psi(t,x)\,dx \, ds \, , \label{inter 15}\\ 
   m_2(t) \int_{\mathbb{R}^n}  \partial_t v(t,x) &  \Psi(t,x) \, dx  \geqslant \varepsilon\, \frac{m_2(0)}{2} \int_{\mathbb{R}^n}v_1(x) \Phi(x) \, dx \notag  \\ & \quad  +\frac{1}{2} \int_0^t m_2(s) \int_{\mathbb{R}^n}|\partial_t u(s,x)|^q\Psi(t,x)\,dx \, ds \, . \label{inter 16}
\end{align}

After the above preparatory results we can finally introduce the functionals whose dynamic is studied in order to prove Theorem \ref{Thm blowup |vt|^p, |ut|^q}. Let us define for any $t\geqslant 0$
\begin{align}
\mathcal{F}(t)& \doteq \varepsilon\, \frac{m_1(0)}{2} \int_{\mathbb{R}^n}\!  u_1(x) \Phi(x) \, dx +\frac{1}{2} \int_0^t \! m_1(s) \int_{\mathbb{R}^n} \! |\partial_t v(s,x)|^p\Psi(t,x)\,dx \, ds \, , \label{def mathcalF}\\
\mathcal{G}(t)& \doteq \varepsilon \frac{m_2(0)}{2} \int_{\mathbb{R}^n}\! v_1(x) \Phi(x) \, dx +\frac{1}{2} \int_0^t\! m_2(s) \int_{\mathbb{R}^n}\! |\partial_t u(s,x)|^q\Psi(t,x)\,dx \, ds \, . \label{def mathcalG}
\end{align}
In particular, \eqref{inter 15} and \eqref{inter 16} may be rewritten as
\begin{align}
 m_1(t) \int_{\mathbb{R}^n}  \partial_t u(t,x)  \Psi(t,x) \, dx & \geqslant \mathcal{F}(t) , \label{inter 15 bis}\\ 
   m_2(t) \int_{\mathbb{R}^n}  \partial_t v(t,x)  \Psi(t,x) \, dx & \geqslant \mathcal{G}(t) . \label{inter 16 bis}
\end{align}
Using H\"older's inequality, we have
\begin{align}
\int_{\mathbb{R}^n}  \partial_t v(t,x)  \Psi(t,x) \, dx \, &\leqslant \bigg(\int_{\mathbb{R}^n}|\partial_t v(t,x)|^p  \Psi(t,x) \, dx\bigg)^{\frac{1}{p}}\bigg(\int_{B_{R+t}}\!\!\Psi(t,x)\, dx\bigg)^{\frac{1}{p'}}\notag \\
\, &\lesssim  (1+t)^{\frac{n-1}{2p'}}\bigg(\int_{\mathbb{R}^n}|\partial_t v(t,x)|^p  \Psi(t,x) \, dx\bigg)^{\frac{1}{p}}, \label{inter holder}
\end{align} where in the last step we used 
\begin{align*}
\int_{B_{R+t}}\Psi(t,x)\, dx \lesssim (1+t)^{\frac{n-1}{2}}.
\end{align*} For further details on this estimate see \cite{YZ06} or \cite[estimate (3.5)]{LT18Glass}. 

Combing \eqref{boundedness multipliers}, \eqref{inter 16 bis} and \eqref{inter holder}, we finally get
\begin{align*}
\mathcal{F}'(t) &=\tfrac{1}{2}m_1(t)\int_{\mathbb{R}^n}|\partial_t v(t,x)|^p\Psi(t,x) \, dx  \\ & \gtrsim m_1(t)\, (1+t)^{-\frac{n-1}{2}(p-1)} \bigg(\int_{\mathbb{R}^n}  \partial_t v(t,x)  \Psi(t,x) \, dx\bigg)^p  \\
& \gtrsim m_1(t)\,  (m_2(t))^{-p}\, (1+t)^{-\frac{n-1}{2}(p-1)} (\mathcal{G}(t))^p  \gtrsim  (1+t)^{-\frac{n-1}{2}(p-1)} (\mathcal{G}(t))^p.
\end{align*}
In a similar way, it is possible to show that \eqref{inter 15 bis} implies the estimate $\mathcal{G}'(t) \gtrsim  (1+t)^{-\frac{n-1}{2}(q-1)} (\mathcal{F}(t))^q$.

Summarizing, throughout this section we proved the following lemma.
\begin{lem} Let us assume that $u_0,u_1,v_0,v_1$ satisfy the assumption of Theorem \ref{Thm blowup |vt|^p, |ut|^q}. Let $(u,v)$ be a local solution of \eqref{weakly coupled system} and let $\mathcal{F}$ and $\mathcal{G}$ be the functionals defined by \eqref{def mathcalF} and \eqref{def mathcalG}, respectively. Then, the following estimates hold:
\begin{align}
\mathcal{F}'(t) &\geqslant C  (1+t)^{-\frac{n-1}{2}(p-1)} (\mathcal{G}(t))^p \label{fundamental estimate F'},\\
\mathcal{G}'(t) &\geqslant K  (1+t)^{-\frac{n-1}{2}(q-1)} (\mathcal{F}(t))^q, \label{fundamental estimate G'}
\end{align}
for any $t\geqslant 0$, where $C,K$ are positive constants depending on $n,p,q,R,b_1,b_2$.
\end{lem}

\begin{rem} In some cases it is more convenient to rewrite \eqref{fundamental estimate F'} and \eqref{fundamental estimate G'} in the integral form, namely,
\begin{align}
\mathcal{F}(t) &\geqslant \mathcal{F}(0) +C \int_0^t (1+s)^{-\frac{n-1}{2}(p-1)} (\mathcal{G}(s))^p \, ds \label{fundamental estimate F},\\
\mathcal{G}(t) &\geqslant \mathcal{G}(0)+K  \int_0^t (1+s)^{-\frac{n-1}{2}(q-1)} (\mathcal{F}(s))^q \, ds, \label{fundamental estimate G}
\end{align} for any $t\geqslant 0$.
\end{rem}

\section{Proof of Theorem \ref{Thm blowup |vt|^p, |ut|^q}: subcritical case} \label{Section |vt|^p, |ut|^q subcrical case}

In this section we prove the blow-up result in the subcritical case, that is, for $p,q>1$ satisfying $$\max\left\{\frac{p+1}{pq-1},\frac{q+1}{pq-1}\right\}> \frac{n-1}{2}\,.$$

The main tool of the proof is the next Kato's type lemma on ordinary differential inequalities including an upper bound estimate for the lifespan, whose proof can be found in \cite{Tak15}.

\begin{lem} \label{Kato's lemma} Let $r>1$, $a>0$, $b>0$ satisfy $$M\doteq \frac{r-1}{2}a-\frac{b}{2}+1>0.$$ Assume that $H\in \mathcal{C}^2([0,T))$ satisfies
\begin{align}
 & H(t)  \geqslant A t^a  \  \  \, \qquad \qquad  \qquad \qquad \mbox{for} \ \ t\geqslant T_0, \label{H lower bound Kato lemma} \\
 & H''(t)  \geqslant B (t+R)^{-b}|H(t)|^{r} \qquad \mbox{for} \ \ t\geqslant 0, \label{H'' lower bound Kato lemma} \\
&  H(0)  \geqslant 0,  \ \ H'(0)>0, \label{H(0), H'(0) conditions Kato lemma}
\end{align} where $A,B,R,T_0$ are positive constants. Then, there exists a positive constant $C_0=C_0(r,a,b,B)$ such that 
\begin{align} 
T< 2^{\frac{2}{M}}T_1 \label{upper bound T Kato lemma}
\end{align} 
 holds, provided that
\begin{align}
T_1\doteq \max\left\{T_0,\frac{H(0)}{H'(0)},R\right\} \geqslant C_0 A^{-\frac{r-1}{2M}}. \label{lower bound T1 Kato lemma}
\end{align}
\end{lem}

Let us consider the case in which $p,q$ satisfy $\frac{p+1}{pq-1}-\frac{n-1}{2}>0$.
From \eqref{fundamental estimate G} and H\"older's inequality, it follows
\begin{align*}
\mathcal{G}(t) & \gtrsim  (1+t)^{-\frac{n-1}{2}(q-1)}\int_0^t (\mathcal{F}(s))^q\, ds \\ &\gtrsim   (1+t)^{-\frac{n-1}{2}(q-1)-(q-1)} \bigg(\int_0^t \mathcal{F}(s)\, ds \bigg)^q.
\end{align*} Plugging this lower bound for $\mathcal{G}$ in \eqref{fundamental estimate F'}, we get
\begin{align}
\mathcal{F}'(t) \gtrsim (1+t)^{-\frac{n-1}{2}(pq-1)-p(q-1)}  \bigg(\int_0^t \mathcal{F}(s)\, ds \bigg)^{pq}. \label{lower bound mathcalF'}
\end{align}
Let us define now the functional $$F(t)\doteq \int_0^t \mathcal{F} (\tau) d\tau.$$ Then, \eqref{lower bound mathcalF'} is equivalent to 
\begin{align} \label{F'' lower bound Kato lemma}
F''(t) \gtrsim (1+t)^{-\frac{n-1}{2}(pq-1)-p(q-1)}  (F(t))^{pq} \qquad \mbox{for} \ \ t\geqslant 0.
\end{align} Besides, 
\begin{align} \label{F(0), F'(0) conditions Kato lemma}
F(0)=0, \qquad F'(0)=\mathcal{F}(0)=\varepsilon I_1[u_1]>0,
\end{align} where $I_1[u_1]\doteq \frac{m_1(0)}{2} \int_{\mathbb{R}^n}u_1(x) \Phi(x) \, dx$. Finally, since $\mathcal{F}$ is increasing,
\begin{align} F(t)\geqslant t \, \mathcal{F}(0) = A_0 \varepsilon \, t \qquad \mbox{for} \ \ t\geqslant 0,\label{F lower bound Kato lemma} 
\end{align} where $A_0$ is positive, independent of $\varepsilon$ constant. Combining \eqref{F'' lower bound Kato lemma}, \eqref{F(0), F'(0) conditions Kato lemma} and \eqref{F lower bound Kato lemma}, we can apply Lemma \ref{Kato's lemma} to $F$ with $r=pq$, $a=1$, $b=\frac{n-1}{2}(pq-1)+p(q-1)$, $A=A_0 \varepsilon$ and $R=1$. In particular, thanks to \eqref{F lower bound Kato lemma}, we can take $T_0\doteq (A_0 \varepsilon)^{-\Lambda(n,p,q)^{-1}}$. 

Therefore,  we may choose $\varepsilon_0$ sufficiently small, such that for any $\varepsilon\in (0,\varepsilon_0]$ the condition $$T_0\geqslant \max\left\{\frac{F(0)}{F'(0)},R\right\}$$ holds, due to the fact that the quantity $\frac{F(0)}{F'(0)}=0$ does not depend on  $\varepsilon$, and, then, using the notations of Lemma \ref{Kato's lemma}, we have $T_0=T_1$.
  Hence, \eqref{upper bound T Kato lemma} implies $T\lesssim \varepsilon^{-\Lambda(n,p,q)^{-1}}$. 
  
  The treatment of the case in which $p,q$ satisfy $\frac{q+1}{pq-1}-\frac{n-1}{2}>0$ is totally symmetric. Indeed, by switching the role of  $\mathcal{F}$ and $\mathcal{G}$ we get $T\lesssim \varepsilon^{-\Lambda(n,q,p)^{-1}}$. This completes the proof of \eqref{lifespan upper bound estimate} in the subcritical case.

\section{Proof of Theorem \ref{Thm blowup |vt|^p, |ut|^q}: critical case} \label{Section |vt|^p, |ut|^q crical case}

In this section we prove the blow-up result in the critical case $$\max\left\{\frac{p+1}{pq-1},\frac{q+1}{pq-1}\right\}= \frac{n-1}{2}\,.$$  Differently from Section \ref{Section |vt|^p, |ut|^q subcrical case} in this case we will employ an iteration argument. Without loss of generality we may assume $\tfrac{p+1}{pq-1}=\tfrac{n-1}{2}$ (in the case $\tfrac{q+1}{pq-1}=\tfrac{n-1}{2}$ the proof is completely analogous, provided that we switch the roles of $\mathcal{F}$ and $\mathcal{G}$).

 In order to get \eqref{lifespan upper bound estimate} in critical case, we have to consider separately the case $p\neq q$, which corresponds to $\Lambda(n,p,q)=0<\Lambda(n,q,p)$, from the case $p=q$, which corresponds to $\Lambda(n,p,q)=0=\Lambda(n,q,p)$.

\subsection*{Case $p\neq q$} \label{Subsection crit p not= q}

In this case we apply the so-called slicing method (cf. \cite{AKT00}, where this approach has been used for the first time). We introduce the sequence $\{\ell_j\}_{j\in \mathbb{N}}$ with $\ell_j\doteq 2-2^{-j}$. The first step of our procedure consists in proving via an inductive argument the sequence of lower bound estimates
\begin{align}\label{lower bound mathcalF j p noteq q}
\mathcal{F}(t)\geqslant C_j \bigg(\log \bigg(\frac{t}{\ell_j}\bigg)\bigg)^{a_j} \qquad \mbox{for any} \ t\geqslant \ell_j  \  \mbox{and} \ \mbox{for any} \ j\in \mathbb{N},
\end{align} where $\{a_j\}_{j\in\mathbb{N}}$ and $\{C_j\}_{j\in\mathbb{N}}$ are sequences of nonnegative numbers that we will determine throughout this section. We point out that, due to \eqref{fundamental estimate F}, \eqref{lower bound mathcalF j p noteq q} is satisfied  in the case $j=0$ with $$a_0\doteq 0 \ \ \mbox{and} \ \ C_0\doteq \mathcal{F}(0)=\varepsilon\, I_1[u_1].$$ 

Let us prove now the inductive step. We assume that \eqref{lower bound mathcalF j p noteq q} is true for some $j\geqslant 0$. Hence, plugging \eqref{lower bound mathcalF j p noteq q} in \eqref{fundamental estimate G} and shrinking the domain of integration, we have for $t\geqslant \ell_{j+1}\geqslant 1$
\begin{align*}
\mathcal{G}(t) & \geqslant K \int_{\ell_j}^t (1+s)^{-\frac{n-1}{2}(q-1)} C_j^q \Big(\log \left(\tfrac{s}{\ell_j}\right)\Big)^{a_jq}  ds \\  & \geqslant K C_j^q (1+t)^{-\frac{n-1}{2}(q-1)}  \int_{\ell_j}^t  \Big(\log \left(\tfrac{s}{\ell_j}\right)\Big)^{a_jq} ds \\
 & \geqslant K C_j^q (1+t)^{-\frac{n-1}{2}(q-1)}  \int_{\tfrac{\ell_j t}{\ell_{j+1}}}^t  \Big(\log \left(\tfrac{s}{\ell_j}\right)\Big)^{a_jq} ds \\ & \geqslant K C_j^q \Big(1-\tfrac{\ell_j }{\ell_{j+1}}\Big) t \,(1+t)^{-\frac{n-1}{2}(q-1)}  \Big(\log \left(\tfrac{t}{\ell_{j+1}}\right)\Big)^{a_jq}   \\
 & \geqslant  K C_j^q 2^{-(j+3)} (1+t)^{-\frac{n-1}{2}(q-1)+1}  \Big(\log \left(\tfrac{t}{\ell_{j+1}}\right)\Big)^{a_jq}, 
\end{align*} where in the last step we employed the inequality $1-\frac{\ell_j}{\ell_{j+1}}\geqslant 2^{-(j+2)}$.

Using the above lower bound for $\mathcal{G}$ in \eqref{fundamental estimate F}, after restricting the domain of integration, for $t\geqslant \ell_{j+1}$ we arrive at
\begin{align*}
\mathcal{F}(t) & 
\geqslant C K^p C_j^{pq} 2^{-(j+3)p}  \int_{\ell_{j+1}}^t (1+s)^{-\frac{n-1}{2}(pq-1)+p}  \Big(\log \left(\tfrac{s}{\ell_{j+1}}\right)\Big)^{a_jpq} ds \\
& \geqslant C K^p C_j^{pq} 2^{-(j+3)p-1}  \int_{\ell_{j+1}}^t s^{-1}  \Big(\log \left(\tfrac{s}{\ell_{j+1}}\right)\Big)^{a_jpq} ds \\
& = C K^p C_j^{pq} 2^{-(j+3)p-1} (a_j pq+1)^{-1}  \Big(\log \left(\tfrac{t}{\ell_{j+1}}\right)\Big)^{a_jpq+1},
\end{align*} where in the second last step we used the condition $\Lambda(n,p,q)=0$ to get $-1$ as power for term $(1+s)$ in the integral. Summarizing we proved \eqref{lower bound mathcalF j p noteq q} for $j+1$ with 
\begin{align*}
C_{j+1}\doteq C K^p  2^{-(j+3)p-1} (a_j pq+1)^{-1} C_j^{pq} \ \ \mbox{and} \ \ a_{j+1}\doteq a_j pq +1.
\end{align*}

In the next step we determine a lower bound for $C_j$. However, we need to find the explicit representation of $a_j$ first. As $a_j=1+ pq a_{j-1}$, applying iteratively this relation and the value of initial element of the sequence $a_0=0$, we find
\begin{align}
a_j=\sum_{k=0}^{j-1} (pq)^k +(pq)^j a_0= \frac{(pq)^j-1}{pq-1}. \label{explicit expression aj p noteq q}
\end{align}
 In particular, $a_{j-1} pq+1 =a_j\leqslant (pq-1)^{-1} (pq)^j$ implies 
\begin{align} \label{lower bound Cj p noteq q}
C_j\geqslant N 2^{-jp}(pq)^{-j} C_{j-1}^{pq},
\end{align} where $N\doteq  C K^p  2^{-2p-1}(pq-1)$. Applying the  logarithmic function to both sides of \eqref{lower bound Cj p noteq q} and using iteratively the obtained inequality, we find
\begin{align}
\log C_j  & \geqslant pq \log C_{j-1}  -j\log (2^{p}(pq))+\log N \notag \\
& \geqslant (pq)^2 \log C_{j-2}  -(j+(j-1)pq)\log (2^{p}(pq))+(1+pq)\log N \notag \\
& \geqslant \cdots \geqslant (pq)^j \log C_{0}  -\sum_{k=0}^{j-1} (j-k)(pq)^k\log (2^{p}(pq))+\sum_{k=0}^{j-1}(pq)^k \log N \notag \\
&= (pq)^j \bigg(\log C_0-\frac{pq}{(pq-1)^2}\log(2^p pq )+\frac{\log N}{pq-1}\bigg) \notag \\ & \quad +(j+1)\,\frac{\log(2^p pq)}{pq-1}+\frac{\log(2^p pq)}{(pq-1)^2}-\frac{\log N}{pq-1}, \label{lower bound log Cj p noteq q}
\end{align} where in the last step we employed the formula
\begin{align}\label{sum formulas}
\sum_{k=0}^{j-1} (j-k)(pq)^k= \frac{1}{pq-1} \bigg(\frac{(pq)^{j+1}-1}{pq-1}-(j+1)\bigg),
\end{align} which can be proved by induction.

Thus, for $j\geqslant j_0\doteq \lceil \frac{\log N}{\log (2^p pq)}-1-\frac{1}{pq-1} \rceil$ by \eqref{lower bound log Cj p noteq q} we get 
\begin{align}
\log C_j  & \geqslant  (pq)^j \bigg(\log C_0-\frac{pq}{(pq-1)^2}\log(2^p pq )+\frac{\log N}{pq-1}\bigg) \notag \\ &= (pq)^j \log (D \varepsilon) ,\label{lower bound log Cj p noteq q 2}
\end{align} where $D\doteq (2^p pq)^{-\frac{pq}{(pq-1)^2}}N^{\frac{1}{pq-1}}I_1[u_1]$. Combining \eqref{lower bound mathcalF j p noteq q} and \eqref{lower bound log Cj p noteq q 2}, it results for any $t\geqslant 2 \geqslant \ell_j$
\begin{align*}
\mathcal{F}(t) & \geqslant \exp\big((pq)^j \log(D\varepsilon)\big) \Big(\log\big(\tfrac{t}{\ell_j}\big)\Big)^{a_j}   \geqslant \exp\big((pq)^j \log(D\varepsilon)\big) \Big(\log\big(\tfrac{t}{2}\big)\Big)^{a_j}.
\end{align*} Since for any $t\geqslant 4$ it holds $\log\big(\frac{t}{2}\big)\geqslant \frac{1}{2}\log t$, from the above relation and \eqref{explicit expression aj p noteq q} it follows
\begin{align}
\mathcal{F}(t) & \geqslant \Big((pq)^j \log\Big(2^{-\frac{1}{pq-1}}D\varepsilon\, (\log t)^{\frac{1}{pq-1}}\Big)\Big) \big(\log\big(\tfrac{t}{2}\big)\big)^{-\frac{1}{pq-1}}. \label{lower bound mathcalF final one}
\end{align}

Finally, we may choose $\varepsilon_0>0$ so small that $$\exp\left(2 D^{-pq+1}\varepsilon_0^{-(pq-1)}\right)\geqslant 4.$$ Consequently, for any $\varepsilon\in (0,\varepsilon_0]$ and for $t >\exp\left(2 D^{-pq+1}\varepsilon^{-(pq-1)}\right) $ it results $t\geqslant 4$ and $\log\Big(2^{-\frac{1}{pq-1}}D\varepsilon\, (\log t)^{\frac{1}{pq-1}}\Big)>0$ and, thus, letting $j\to \infty$ in \eqref{lower bound mathcalF final one} we find that the lower bound of $\mathcal{F}(t)$ blows up. Therefore, $\mathcal{F}(t)$ may be finite only for $t\leqslant \exp\left(2D^{-pq+1}\varepsilon^{-(pq-1)}\right)$. This is exactly \eqref{lifespan upper bound estimate} in the critical case $\Lambda(n,p,q)=0$ for $p\neq q$.

\subsection*{Case $p=q$} \label{Subsection crit p = q}

In this section we consider the case $\Lambda(n,p,q)=0=\Lambda(n,q,p)$. In particular, we have $p=q$. Moreover, the condition $\Lambda(n,p,p)=0$ is satisfied if and only if $p=p_{\Gla}(n)=\frac{n+1}{n-1}$. This implies that the powers of $(1+s)$ in the right hand sides of \eqref{fundamental estimate F'} and \eqref{fundamental estimate G'} are exactly $-1$. Therefore, up to a not relevant modification of the multiplicative constants, we may reformulate the integral version of \eqref{fundamental estimate F'} and \eqref{fundamental estimate G'} as follows:
\begin{align}
\mathcal{F}(t) &\geqslant \mathcal{F}(0) +C \int_1^t  \frac{(\mathcal{G}(s))^p}{s} \, ds \label{fundamental estimate F p=q},\\
\mathcal{G}(t) &\geqslant \mathcal{G}(0)+K  \int_1^t\frac{(\mathcal{F}(s))^q}{s} \, ds. \label{fundamental estimate G p=q}
\end{align}

Due to the particular structure of the iteration frame given by \eqref{fundamental estimate F p=q}, \eqref{fundamental estimate G p=q}, it is not necessary to slice the time interval in this case. As in the previous section, the first step is to prove the lower bound estimates
\begin{align}\label{lower bound mathcalF j p = q}
\mathcal{F}(t)\geqslant C_j \big(\log t \big)^{a_j} \qquad \mbox{for any} \ t\geqslant 1  \  \mbox{and} \ \mbox{for any} \ j\in \mathbb{N},
\end{align} where $\{a_j\}_{j\in\mathbb{N}}$ and $\{C_j\}_{j\in\mathbb{N}}$ are sequences of nonnegative numbers that we will determine throughout the proof. According to \eqref{fundamental estimate F p=q}, we see that \eqref{lower bound mathcalF j p = q} is true for $j=0$ provided that $a_0\doteq 0$ and $C_0\doteq \varepsilon I_1[u_1]$. Let us prove now the inductive step. Plugging the lower bound \eqref{lower bound mathcalF j p = q} in \eqref{fundamental estimate G p=q}, we find for $t\geqslant 1$
\begin{align*}
\mathcal{G}(t) &\geqslant KC_j^q  \int_1^t\frac{\big(\log s \big)^{a_jq} }{s}   \, ds = KC_j^q  (a_jq+1)^{-1} \big(\log t \big)^{a_jq+1}.
\end{align*} Combining the last inequality with \eqref{fundamental estimate F p=q}, we have
\begin{align*}
\mathcal{F}(t) &\geqslant C  K^pC_j^{pq}  (a_jq+1)^{-p} \int_1^t  \frac{ \big(\log s \big)^{a_jpq+p}}{s} \, ds \\ &=  C  K^pC_j^{pq}  (a_jq+1)^{-p}  (a_jpq+p+1)^{-1}\big(\log t \big)^{a_jpq+p+1}.
\end{align*} Also, we proved \eqref{lower bound mathcalF j p = q} for $j+1$, provided that 
$$a_{j+1}\doteq a_j pq+p+1 \quad \mbox{and} \quad C_{j+1}\doteq  C  K^pC_j^{pq}  (a_jq+1)^{-p}  (a_jpq+p+1)^{-1}.$$ By iteration, we arrive at
\begin{align}
a_j=(p+1)\sum_{k=0}^{j-1} (pq)^k +(pq)^j a_0= \frac{p+1}{pq-1}\big((pq)^j-1\big)= \frac{(pq)^j-1}{p-1}. \label{explicit expression aj p = q}
\end{align}

By \eqref{explicit expression aj p = q} we derive, in particular, 
 \begin{align*}
 a_{j-1} pq+p+1 &=a_j\leqslant (p-1)^{-1} (pq)^j, \\
 a_{j-1}q+1 &= \tfrac{(pq)^j}{p(p-1)}-\tfrac{1}{p-1}\leqslant (p(p-1))^{-1}(pq)^j.
 \end{align*} The above estimates imply 
\begin{align} \label{lower bound Cj p = q}
C_j \geqslant N \big((pq)^{p+1}\big)^{-j} C_{j-1}^{pq},
\end{align} 
where $N\doteq  C K^p  (p-1)^{p+1}p^p$. Applying the logarithm function to both sides of \eqref{lower bound Cj p = q} and using in an iterative way the resulting relation, we have
\begin{align}
\log C_j  & \geqslant pq \log C_{j-1}  -j\log (pq)^{p+1}+\log N \notag \\
& \geqslant (pq)^2 \log C_{j-2}  -(j+(j-1)pq)\log (pq)^{p+1}+(1+pq)\log N \notag \\
& \geqslant \cdots \geqslant (pq)^j \log C_{0}  -\sum_{k=0}^{j-1} (j-k)(pq)^k\log (pq)^{p+1}+\sum_{k=0}^{j-1}(pq)^k \log N \notag \\
&= (pq)^j \bigg(\log C_0-\frac{pq}{(pq-1)^2}\log(pq)^{p+1}+\frac{\log N}{pq-1}\bigg) \notag \\ & \quad +(j+1)\,\frac{\log(pq)^{p+1}}{pq-1}+\frac{\log(pq)^{p+1}}{(pq-1)^2}-\frac{\log N}{pq-1}, \label{lower bound log Cj p = q}
\end{align} where we used again \eqref{sum formulas}.

Thus, for $j\geqslant j_0\doteq \lceil \frac{\log N}{\log (pq)^{p+1}}-1-\frac{1}{pq-1} \rceil$ by \eqref{lower bound log Cj p = q} we get 
\begin{align}
\log C_j  & \geqslant  (pq)^j \bigg(\log C_0-\frac{pq}{(pq-1)^2}\log(pq)^{p+1}+\frac{\log N}{pq-1}\bigg)\notag \\ &= (pq)^j \log (D \varepsilon), \label{lower bound log Cj p = q 2}
\end{align} where $D\doteq (pq)^{-\frac{pq(p+1)}{(pq-1)^2}}N^{\frac{1}{pq-1}}I_1[u_1]$. Then, using together \eqref{lower bound mathcalF j p = q}, \eqref{explicit expression aj p = q} and \eqref{lower bound log Cj p = q 2},  $t\geqslant 1$ it holds
\begin{align}
\mathcal{F}(t) & \geqslant \exp\big((pq)^j \log(D\varepsilon)\big) \big(\log t \big)^{a_j}  \notag  \\ &= \exp\Big((pq)^j \log\Big(D\varepsilon \, \big(\log t\big)^{\frac{1}{p-1}}\Big)\Big) \big(\log t \big)^{-\frac{1}{p-1}}. \label{lower bound mathcalF final one p=q}
\end{align} 

In conclusion, we can find $\varepsilon_0>0$ sufficiently small, so that $$\exp\left(D^{-p+1}\varepsilon_0^{-(p-1)}\right)\geqslant 1.$$ So, for any $\varepsilon\in (0,\varepsilon_0]$ and for $t >\exp\left(D^{-p+1}\varepsilon^{-(p-1)}\right) $ we have $t\geqslant 1$ and $$\log\Big(D\varepsilon \, \big(\log t\big)^{\frac{1}{p-1}}\Big)>0.$$
 However, taking the limit as $j\to \infty$ in \eqref{lower bound mathcalF final one p=q}, we find that the lower bound of $\mathcal{F}(t)$ blows up.
  Hence, $\mathcal{F}(t)$ can be  finite only for $t\leqslant \exp\left(D^{-p+1}\varepsilon^{-(p-1)}\right)$. This is precisely \eqref{lifespan upper bound estimate} in the critical case $p= q=p_{\Gla}(n)$.

\subsection*{Acknowledgment}

The first author is member of the Gruppo Nazionale per L'Analisi Matematica, la Probabilit\`{a} e le loro Applicazioni (GNAMPA) of the Instituto Nazionale di Alta Matematica (INdAM). This paper was written partially during the stay of the first author at Tohoku University within the period October to December 2018. He thanks the Mathematical Department of Tohoku University for the hospitality and the great working conditions during his stay. The second author is partially supported by the Grant-in-Aid for Scientific Research (B)(No.18H01132).


\end{document}